\documentclass{amsart}

\usepackage{amsmath,amssymb,amsfonts,url,epsf,epsfig}
\usepackage{euscript,graphicx,float}

\usepackage{hyperref}

\newtheorem{theorem}{Theorem}[section]
\newtheorem{lemma}[theorem]{Lemma}

\newtheorem{corollary}[theorem]{Corollary}

\newcommand{\Aut}{{\rm Aut}}

\newcommand{\Sym}{\hbox{{\rm Sym}}}

\newcommand{\ZZ}{\mathbb{Z}}

\newcommand{\SBP}{\hbox{{\rm SBP}}}
\newcommand{\CDC}{\hbox{{\rm CDC}}}

\newcommand{\alt}{{\rm{alt}}}

\newcommand{\Dart}{{\rm D}}

\newcommand{\PGammaL}{{\rm P}\Gamma{\rm L}}

\newcommand{\la}{\langle}
\newcommand{\ra}{\rangle}

\newcommand{\V}{{\mathcal{V}}}
\newcommand{\D}{{\mathcal{D}}}

\newcommand{\U}{{\rm{U}}} 

\renewcommand{\wr}{\mathop{\hbox{wr}}}

\newcommand{\A}{{\mathcal{A}}}

\newcommand{\AAG}{\mathrm {A^2G}}
\newcommand{\DCyc}{\mathrm {DCyc}}
\newcommand{\dAAG}{\mathrm {A^2D}}

\begin{document}

\title[Separated Box Product]{The Separated Box Product of Two Digraphs}

\author[P.\ Poto\v{c}nik]{Primo\v{z} Poto\v{c}nik}
\address{Primo\v{z} Poto\v{c}nik,\newline
Faculty of Mathematics and Physics, University of Ljubljana, \newline Jadranska 19, SI-1000 Ljubljana, Slovenia;
\newline also affiliated with: \newline 
IAM, University of Primorska,\newline
 Muzejski trg 2, SI-6000 Koper, Slovenia;
 \newline and: \newline
 IMFM,
 Jadranska 19, SI-1000 Ljubljana, Slovenia.
 } 
\email{primoz.potocnik@fmf.uni-lj.si}

\thanks{Supported in part by the Slovenian Research Agency, projects J1-5433, J1-6720, and P1-0294}


\author[S.\ Wilson]{Steve Wilson}
\address{Steve Wilson,\newline
 Northern Arizona University, Department of Mathematics and Statistics, \newline
 Box 5717, Flagstaff, AZ 86011, USA
 \newline also affiliated with: \newline 
FAMNIT, University of Primorska,\newline
 Glagolja\v{s}ka 8, SI-6000 Koper, Slovenia; }
 \email{stephen.wilson@nau.edu}


\begin{abstract}
A new product construction of graphs and digraphs, based on the standard box product of graphs and called
the separated box product, is presented, and several of its properties are discussed. 
Questions about the symmetries of the product and their relations to symmetries of the factor graphs are considered.   An application of this construction to the case of
tetravalent edge-transitive graphs is discussed in detail.
\end{abstract}

\maketitle


\section{Introduction}
\label{sec:intro}

A {\em dart} (also known as an {\em arc}) of a graph $\Gamma$ 
is an ordered pair of adjacent vertices of $\Gamma$, and the set of
all darts of $\Gamma$ is denoted by $\D(\Gamma)$. A graph $\Gamma$ is {\em dart-transitive}
provided that its symmetry group $\Aut(\Gamma)$ acts transitively on $\D(\Gamma)$.

It was proved in \cite{PSV} that
apart from a well-understood infinite family of graphs and finitely many exceptions,
the order of the vertex-stabiliser $G_v$ in the symmetry group $G$ of 
a connected tetravalent dart-transitive graph of order $n$
satisfies the inequality $n\ge 2|G_v|\log_2(|G_v|/2)$.
The largest amongst the finite set of exceptions is a graph (let's denote it by 
$\Gamma_{8100}$) of order $8100$ (appearing in the last line of \cite[Table 2]{PSV})
whose symmetry group is isomorphic to the group $\PGammaL(2,9)\, \hbox{wr}\, C_2$,
and has the vertex-stabiliser of  large order,  $512$. 

This graph  was originally constructed
as a coset graph of the group $\PGammaL(2,9)\, \hbox{wr}\, C_2$, but an interesting combinatorial
construction of $\Gamma_{8100}$ was also given in \cite[Section 2.2]{PSV}: For a  graph $\Lambda$,
 let $\AAG(\Lambda)$, the "squared-arc graph" of $\Lambda$, be the graph with vertex-set being $\D(\Lambda) \times \D(\Lambda)$ and with two vertices
 $(x,y), (w,z) \in \D(\Lambda) \times \D(\Lambda)$ adjacent in $\AAG(\Lambda)$ if and only if $y=w$,
 $x=(v_1,v_2)$ and $z=(v_2,v_3)$ for some vertices $v_1,v_2,v_3$ of $\Lambda$ such that $v_1\not = v_3$.
 The graph $\Gamma_{8100}$ is then isomorphic to the graph $\AAG(\Lambda)$, where $\Lambda$ is the
 Tutte $8$-cage (the unique dart-transitive $3$-regular graph on $30$ vertices with girth $8$).
 
 It is not at all obvious in what way (if at all)  the symmetry group of the Tutte's $8$-cage $\Lambda$,
 isomorphic to $\PGammaL(2,9)$, yields the group $\PGammaL(2,9)\, \hbox{wr}\, C_2$ acting on $\AAG(\Lambda)$.

  Our endeavour to understand the relationship between $\Aut(\Lambda)$ and $\Aut(\AAG(\Lambda))$ led
us to a discovery of a surprisingly  simple product operation on digraphs, called the {\em separated box product}, 
which significantly generalises the
 $\AAG$ construction on one hand and explains many symmetries that $\AAG(\Lambda)$ possesses on the other hand.

  %
 
   It is the aim of this paper to present the separated box product construction and some of its properties.
   The construction is described in Section~\ref{ssec:Con}. In Section~\ref{ssec:sym}, the
   symmetry properties of the resulting (di)graph are discussed, and in Section~\ref{ssec:connect}
   the question of its connectedness of is addressed. 
   The relationship between the $\AAG$ construction and the separated box product is explained in Section~\ref{sec:AAG}.  Here, we prove the first of the two major theorems in this paper, Theorem \ref{the:A2G}, which uses the separated box product to explain the symmetry groups of the graphs constructed as above from  sufficiently symmetric cubic graphs. 
   Finally, in Section~\ref{sec:4}, the case where the separated box product yields
   a tetravalent edge-transitive graphs is discussed in  detail.  In this section, we prove   Theorem \ref{the:type}, which determines the possibilities for the symmetry type of the product based on the relationships between the factor digraphs and their reverses.

\section{Definitions}
\label{sec:def}

	In many parts of graph theory, it is more natural to define ideas in terms of directed graphs and then consider a {\em graph} to be just a special case of a digraph.  We will take that approach here.   A {\em digraph} is a pair $\Gamma = (\V, \D)$ in which $\V$ is a finite non-empty collection of things called {\em vertices} 
and $\D$ is a collection of ordered pairs of distinct vertices.  An element $(u,v)$ of $\D$ will be called a {\em dart} with {\em initial}
vertex $u$ and {\em terminal} vertex $v$,
and we will picture it as an arrow leading from $u$ to $v$. We let $\D(\Gamma) = \D$ and $\V(\Gamma) = \V$ in this case.
A {\em $2$-dart} of a digraph $(\V, \D)$ is a pair $(x,y)$ of darts in $\D$ such that the terminal vertex of $x$ coincides
  with the initial vertex of $y$ while the initial vertex of $x$ does not coincide with the terminal vertex of $y$.

The {\em out-neighbourhood} and the 
{\em in-neighbourhood} of a vertex $v\in \V$ are defined as the sets $\{u\in \V : (v,u) \in \D\}$ and
$\{u\in \V : (u,v) \in \D\}$ and denoted $\Gamma^+(v)$ and $\Gamma^-(v)$, respectively.
The cardinalities of $\Gamma^+(v)$ and $\Gamma^-(v)$ are called the {\em out-valence} and
the  {\em in-valence} of $v$, respectively. A vertex of out-valence (in-valence) $0$ is called 
a {\em sink} ({\em source}, respectively). A digraph in which the in-valence as well as  the out-valence of every
vertex is equal some constant $k$ is called {\em $k$-valent}.

   For a digraph $\Gamma=(\V,\D)$ and a subset $\A \subset \D$, let $\A^{-1} = \{(u,v) : (v,u) \in \D\}$, and let
  $\Gamma^{-1}$ be the digraph $(\V,\D^{-1})$, called the {\em reverse of $\Gamma$}.  If  $\Gamma^{-1}$ is isomorphic to $\Gamma$, we say that $\Gamma$ is {\em reversible} and any isomorphism between $\Gamma$ and $\Gamma^{-1}$
  is called a {\em reversal} of $\Gamma$.
A digraph $\Gamma$ such that $\Gamma=\Gamma^{-1}$ is called a {\em graph}. If $(u,v)$ is
a dart of a graph $\Gamma$, then so is $(v,u)$;
 in this case, we call the pair $\{(u,v), (v, u)\}$ an {\em edge} of the graph and denote it by $uv$ (or $vu$).  We picture $uv$ as an  undirected link joining $u$ and $v$.

	At the other extreme, if $\D$ and $\D^{-1}$ are disjoint, we call  $(\V,\D)$ an {\em orientation}.
  The {\em underlying graph} of a digraph $\Gamma=(\V, \D)$ is the graph $(\V, \D \cup \D^{-1})$, denoted $\U\Gamma$.
  
    A digraph is said to be {\em connected}   provided that its underlying graph is connected and it is said to be {\em bipartite}
   if its underlying graph is bipartite. 
  
  The set of all permutations of a given set $\Omega$ will be denoted by $\Sym(\Omega)$.
For $g\in\Sym(\Omega)$ and $\omega\in \Omega$, the image of $\omega$ under $g$ will be written as $\omega^g$,
while the product of two permutations $g,h \in \Sym(\Omega)$ is defined by $\omega^{(gh)} = (\omega^g)^h$.

A {\em symmetry} (or an {\em symmetry}) of a digraph $\Gamma = (\V, \D)$ is a permutation of $\V$ which,
in its natural action on $\V\times \V$, preserves $\D$.
The symmetries of $\Gamma$, together with the above defined product, form a group, denoted $\Aut(\Gamma)$.  
Note that any reversal of $\Gamma$ normalises $\Aut(\Gamma)$ and a product of two reversals is
always a symmetry of $\Gamma$.
The set $\Aut^*(\Gamma)$ of all symmetries and all reversals of $\Gamma$ thus forms a group in which
$\Aut(\Gamma)$ has index at most two.

If $\Gamma$ is a digraph and $G\le \Aut(\Gamma)$ such that
 $G$ is transitive on $\V(\Gamma)$ ($\D(\Gamma)$), then
$\Gamma$ is said to be {\em $G$-vertex-transitive} ({\em $G$-dart-transitive}, respectively).

The  {\em box product}, also known as the {\em Cartesian product} \cite{klavzar} of two digraphs $\Gamma_1$ and $\Gamma_2$,
denoted by $\Gamma_1\Box \Gamma_2$, is the digraph with vertex-set $\V(\Gamma_1)\times\V(\Gamma_2)$ and
a pair $\bigl((v_1,v_2),(u_1,u_2)\bigr)$ of vertices $(v_1,v_2), (u_1,u_2) \in \V(\Gamma_1)\times\V(\Gamma_2)$
being a dart of $\Gamma_1\Box \Gamma_2$ if and only if $(v_1,u_1) \in \D(\Gamma_1)$ and $u_2 = v_2$ (these are the {\em horizontal} edges) or $v_1 = u_1$ and $(v_2,u_2) \in \D(\Gamma_2)$  (these are the {\em vertical} edges). 

\section{The Separated Box Product}
\label{sec:SBP}

\subsection{Construction and basic properties}
\label{ssec:Con}

We define here a construction, the {\em separated box product} of two digraphs, which generalizes that of $\AAG(\Lambda)$.

For digraphs $\Gamma_1 =(\V_1, \D_1)$ and $\Gamma_2 = (\V_2, \D_2)$  
we define a digraph $\Gamma = (\V, \D)$, where $\V = \V_1\times\V_2\times\ZZ_2$, and darts in $\D$ are all $((a, x, 0), (b,x,1))$ where $(a, b) \in\D_1, x\in \V_2$, together with all  $((a, x,1), (a, y, 0))$ where $a \in\V_1, (x,y)\in \D_2$.  We refer to $\Gamma$ 
as $\SBP(\Gamma_1, \Gamma_2)$. 
Figure \ref{fig:BoxP} shows some of the darts in $\Gamma$.

\begin{figure}[hhh]
\begin{center}
\epsfig{file=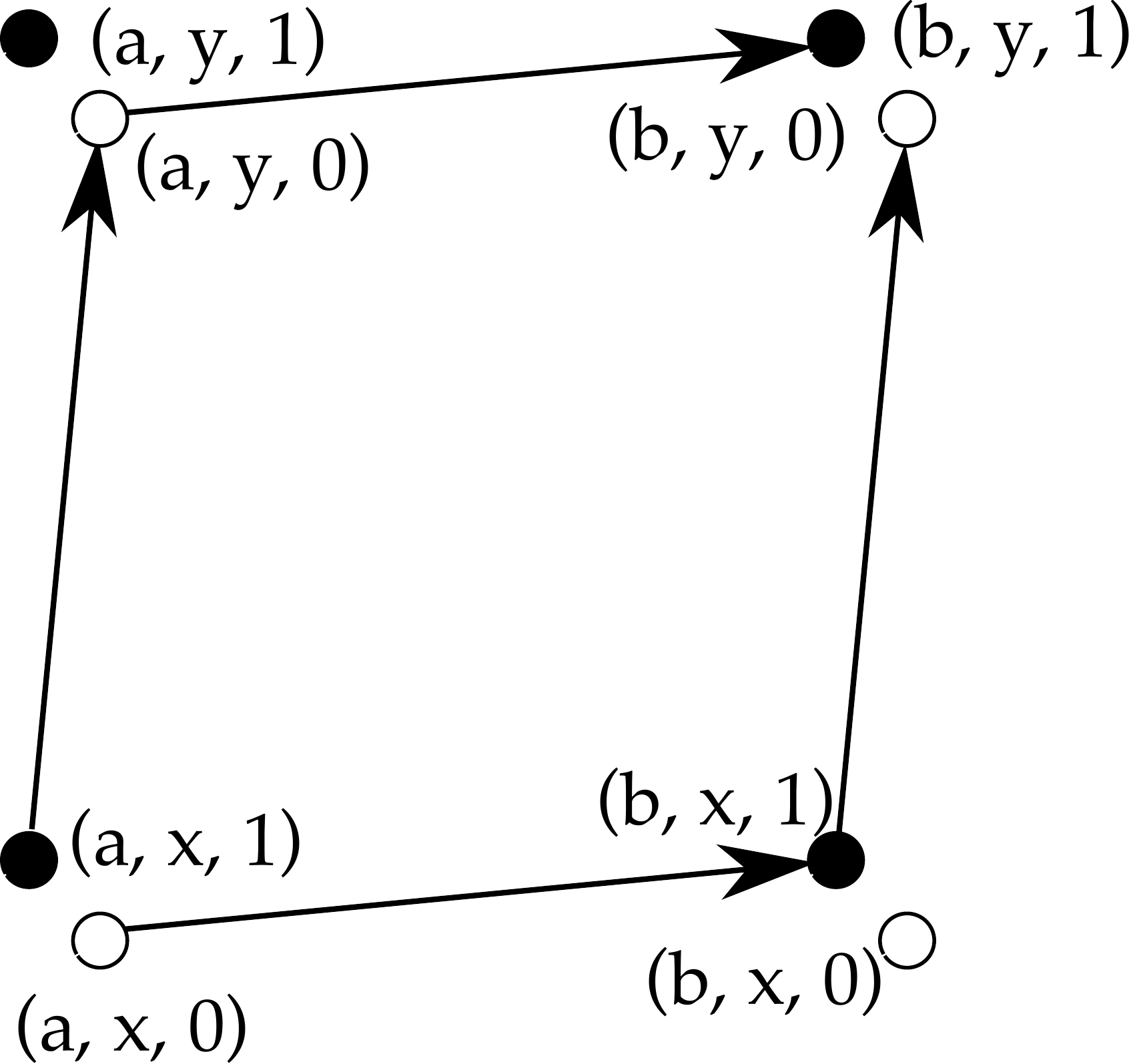,height=40mm}
\end{center}
\caption{Edges in the SBP($\Gamma_1, \Gamma_2$)}
\label{fig:BoxP}
\end{figure}

As we shall see, $\Gamma = \SBP(\Gamma_1, \Gamma_2)$ is a digraph which is often disconnected
(see Section~\ref{ssec:connect}).   If  $\Gamma $ is sufficiently symmetric, the components are isomorphic.  In this case, we will use the notation $\Gamma_1 \#\Gamma_2$ to stand for a single connected component of the 
underlying graph of $\Gamma$.

In accordance with Figure \ref{fig:BoxP}, we refer to vertices with third coordinate 0 as {\em white} vertices and those with third coordinate 1 as {\em black}.  Further, we refer to darts of the first type as {\em horizontal darts}, and those of the second type as {\em vertical darts}, just as we do in the ordinary box product.  

Identifying each $(a, x, 0)$ with $(a, x, 1)$ results in the ordinary box product of the two digraphs, hence the name of the construction.

Since horizontal darts always point from a white vertex to a black vertex, the reverse of a horizontal dart is never a dart of $\Gamma$. Similarly, since vertical darts always point from a black vertex to a white vertex, the reverse of a vertical dart is never a dart. This shows that the separated box product of two digraphs is an orientation.

Furthermore,
 observe that the out- and in-valence of a white vertex $(a,x,0)$ are equal to the out-valence of $a$ in $\Gamma_1$ and
the in-valence of $x$ in $\Gamma_2$, respectively. Similarly, the out- and in-valence of a black vertex $(a,x,1)$ 
equal the out-valence of $x$ in $\Gamma_2$ and the in-valence of $a$ in $\Gamma_1$, respectively.
In particular, if both $\Gamma_1$ and $\Gamma_2$ are $k$-valent, then so is $\SBP(\Gamma_1,\Gamma_2)$.

\subsection{Symmetries of the separated box product}
\label{ssec:sym}

The separated box product clearly inherits some symmetries from its parents.  The following result is straightforward and the
proof is left to the reader.

\begin{lemma}
\label{L2}
Let $\Gamma_1$ and $\Gamma_2$ be digraphs without sources and sinks, 
let $\Gamma=\SBP(\Gamma_1,\Gamma_2)$, and let $G_j \le \Aut(\Gamma_j)$ for $j\in \{1,2\}$.
Then, by letting
$$(a,x,i)^{(g_1,g_2)} = (a^{g_1}, x^{g_2}, i)$$
for every $(a,x,i)$ of $\Gamma$ and every $(g_1,g_2) \in G_1 \times G_2$, the group
 $G_1 \times G_2$ acts faithfully on $\Gamma$ as a group of symmetries.
Moreover:
\begin{itemize}
\item[{\rm (i)}]
If $G_1$ and $G_2$ act transitively on the vertices of the respective digraphs, then $G_1\times G_2$
has two orbits on the vertices of  $\Gamma$, one being the 
set of white vertices and the other the set of  black vertices.
\item[{\rm (ii)}]
If $G_1$ and $G_2$ act transitively on the darts of the respective digraphs, then $G_1\times G_2$
has two orbits on the darts of  $\Gamma$, one being the 
set of horizontal darts and the other the set of vertical darts.
\end{itemize}

If $\Gamma_1$ and $\Gamma_2$ possess  reversals $\sigma_1$ and $\sigma_2$, respectively, then the permutation $\sigma$ of
$\V(\Gamma)$, defined by 
$$
(a,x,i)^{\sigma} = (a^{\sigma_1}, x^{\sigma_2},1-i),
$$
is a reversal of $\Gamma$.
\end{lemma}

Observe that the  reversal $\sigma$ in the above lemma interchanges the color classes of vertices in $\SBP(\Gamma_1,\Gamma_2)$.
Hence an obvious corollary:

\begin{corollary}
\label{cor:LR}
Let $\Gamma_1$ and $\Gamma_2$ be digraphs without sources and sinks,
possessing reversals $\sigma_1$ and $\sigma_2$, 
let $\Gamma=\SBP(\Gamma_1,\Gamma_2)$, and let $G_j \le \Aut(\Gamma_j)$ for $j\in \{1,2\}$ be transitive the darts
of the respective digraphs.
Furthermore, let $G_1\times G_2$ and $\sigma$ act upon $\Gamma$ as in Lemma~\ref{L2}. 

Then $\la G_1\times G_2, \sigma\ra$ acts vertex-transitively as a group of symmetries on the underlying graph $\U\Gamma$ of $\Gamma$
and has two orbits on the darts and on the edges of $\U\Gamma$, these being the sets of
vertical and of horizontal darts (edges) of $\U\Gamma$.
Moreover if both $\Gamma_1$ and $\Gamma_2$ are of valence at least $2$, the stabiliser of a dart in $G_1\times G_2$
is non-trivial.
\end{corollary}

\begin{proof}
By Lemma~\ref{L2}, $G_1\times G_2$ is transitive on horizontal and on vertical darts, as well as on white and on black vertices.
On the other hand, $\sigma$ interchanges the color-classes of vertices (while preserving horizontal edges),
 implying that $\la G_1\times G_2, \sigma\ra$ is vertex-transitive
and that it has 2 orbits on the darts and on the edges of $\U\Gamma$.

 Suppose now that $\Gamma_1$ and $\Gamma_2$ are of valence at least $2$.
Let $\bigl((a,x,0), (b,x,1)\bigr)$ be a horizontal dart of $\Gamma$. Since $\Gamma_2$ is dart-transitive and of valence at least $2$,
there is some nontrivial symmetry $g_x \in G_2$ fixing $x\in\V(\Gamma_2)$. Now let
$g$ be the permutation of $\V(\Gamma)$ mapping a vertex $(u,v,i)$ of $\Gamma$ to the vertex $(u,v^{g_x},i)$. 
Then $g$ is clearly a non-trivial symmetry of $\Gamma$ which fixes the horizontal dart $\bigl((a,x,0), (b,x,1)\bigr)$.
By exchanging the roles of $\Gamma_1$ and $\Gamma_2$ one can prove in the same manner that the stabiliser of any vertical dart is non-trivial.
\end{proof}

 If $\Gamma_2$ is isomorphic to either $\Gamma_1$ or to $\Gamma_1^{-1}$, more symmetries of $\SBP(\Gamma_1,\Gamma_2)$
 and its underlying graph arise naturally. The proof of the following lemma is straightforward.
 
 \begin{lemma}
 \label{lem:DD}
 For a digraph $\Delta$, let $\tau$ and $\mu$ be the permutations of $\V(\Delta) \times \V(\Delta) \times \ZZ_2$ defined by
 $$(a,x,i)^\tau = (x,a,1-i) \quad \hbox{and} \quad (a,x,i)^\mu = (x,a,i).$$
 Then $\tau$ is an symmetry of $\SBP(\Delta,\Delta)$ and $\mu$ is a reversal of $\SBP(\Delta,\Delta^{-1})$.
 %
 %
\end{lemma}
 
 This, together with Lemma~\ref{L2}, yields the following two results:
 	
\begin{corollary}
\label{the:delta2}
Let $\Delta$ be a digraph without sources and sinks admitting a dart-transitive group of symmetries $G$,
let $\Gamma = \SBP(\Delta,\Delta)$, let $G\times G$ act upon $\Gamma$ as in Lemma~\ref{L2}, and let
$\tau$ be as in Lemma~\ref{lem:DD}.
 Then $\langle G\times G, \tau \rangle$ acts transitively on the darts of $\Gamma$ and is isomorphic to the 
wreath product $G\wr C_2 \cong (G\times G) \rtimes C_2$.

Moreover, if $\Delta$ is reversible with a reversal $\tilde{\sigma}$ and $\sigma$ is the reversal of $\Gamma$ 
guaranteed by Lemma~\ref{L2},
acting as
$$
 (a,x,i)^\sigma = (a^{\tilde{\sigma}},x^{\tilde{\sigma}},1-i),
$$
then the group $\langle G\times G, \tau, \sigma \rangle$ acts dart-transitively on the underlying graph $\U\Gamma$.
\end{corollary}

\begin{corollary}
\label{the:SS}
Let $\Delta$ be a digraph without sources and sinks admitting a dart-transitive group of symmetries $G$,
let $\Gamma = \SBP(\Delta,\Delta^{-1})$, let $G\times G$ act upon $\Gamma$ as in Lemma~\ref{L2}, and let
$\mu$ be the reversal of $\Gamma$ guaranteed by Lemma~\ref{lem:DD}.
Then $\langle G\times G, \mu \rangle$ acts edge-transitvely but not vertex-transitively
on the underlying graph $\U\Gamma$.
\end{corollary}

{\bf Remark.} If $\Delta$ in Corollary~\ref{the:SS} is reversible, then $\SBP(\Delta,\Delta^{-1}) \cong \SBP(\Delta,\Delta)$,
and we may apply Corollary~\ref{the:delta2} to conclude that $\Gamma$ is dart-transitive.
%

\subsection{Connectivity of the separated box product}
\label{ssec:connect}

In order to determine the number and size of connected components of $\SBP(\Gamma_1, \Gamma_2)$, we need to introduce some tools pertaining to digraphs, first introduced in \cite[Section 2]{MarPot} and further developed in \cite{norbert} (see also \cite{norbertamc} for further applications of the technique).
The approach we give here, restricting our attention to orientations, is a slightly simplified form of the notions in \cite{MarPot}.
We point out, however, that all the results mentioned in this subsection (including Theorem~\ref{the:comp})
 apply to general digraphs as well.

 A {\em walk} from a vertex $v_0$ to a vertex $v_k$ in an orientation $\Gamma$ 
is a sequence $W = v_0\ldots v_k$ of vertices $v_i$ such that
for each $i\in \{1,\ldots,k\}$ exactly one of $(v_{i-1},v_i)$ or $(v_i,v_{i-1})$ is a dart of $\Gamma$; in the first case we say that $(v_{i-1},v_i)$ is
{\em positive in $W$} and otherwise it is {\em negative in $W$}. A walk is {\em directed} if all of its darts are positive in it.

The {\em sum} of the walk is the difference between the number of positive and
the number of negative darts in the walk. For two vertices $v$ and $u$ we say that they are {\em alter-related} if there exists a walk from $u$ to $v$ of sum $0$. The relation ``is alter related to'' is an equivalence relation, and which we will denote $\alt$. The number of equivalence classes
of $\alt$ on $\V(\Gamma)$ is called the {\em alter-perimeter} of $\Gamma$.

Is $k$ consecutive darts in a walk $W$ are all positive (negative) in $W$, then we say that the walk they induce in {\em positive}
 ({\em negative}, respectively) in $W$. If a walk of sum $0$ is a concatenation of subwalks, all of the same length $k$ for some $k$,
  that are alternatively positive and negative in $W$, then we say that the walk is in a {\em standard form} with {\em tolerance $k$}.
 Note that, if 
If $v_0\ldots v_m$ is a walk of sum $0$ and $(v_i,u)$ or $(u,v_i)$ is a dart of $\Gamma$, 
then the insertion of the walk $v_iuv_i$ results in walk $v_0\ldots v_i u v_i \ldots v_m$ of sum $0$, implying that after insertion of
 an appropriate set of such walks, one obtains a walk  from $v_0$ to $v_m$ of sum $0$ which is in standard form.
 Hence, we have the following:
 
 \begin{lemma}
 \label{lem:standard}
 If $u$ and $v$ are in the same alter class of an orientation $\Gamma$, then for any sufficiently large integer $k$,
  there exists a walk of sum $0$ from $u$ to $v$ 
 in standard form with tolerance $k$.
 \end{lemma}

If $\Gamma$ is an  orientation of alter-perimeter $m$ and has neither sources nor sinks, then the $\alt$-equivalence classes can be indexed by $\ZZ_m$ in such
a way that every dart points from a class indexed $i$ to the class indexed $i+1$ for some $i\in \ZZ_m$; in other words,
 the quotient digraph of $\Gamma$ with respect to $\alt$ is a directed cycle of length $m$. (Here we allow $m=1$ and
 view  the cycle as a digraph with one vertex and a loop attached to it.)
 
 \begin{theorem}
 \label{the:comp}
 Let $\Gamma_1$ and $\Gamma_2$ be two connected orientations without sources and sinks
 of alter-perimeters $s$ and $t$, respectively. Then
 $\SBP(\Gamma_1,\Gamma_2)$ consists of $\gcd(s,t)$ connected components.
 \end{theorem}
 

\begin{proof}
Let $\Gamma=\SBP(\Gamma_1,\Gamma_2)$ and fix a white vertex $(a,x,0)$ of $\Gamma$. Label the alter-classes of $\Gamma_1$ and
$\Gamma_2$ by $\ZZ_s$ and $\ZZ_t$ 
 in such a way that $a$ and $x$ are in the classes labelled $0$ and
 every dart with its initial vertex in class $i$ has its terminal vertex in class $i+1$.
 Further, 
for each $b\in \V(\Gamma_1)$ from the alter-class labelled $i\in \ZZ_s$ and $y\in \V(\Gamma_2)$ from the alter-class labelled $j\in\ZZ_t$,
label the two vertices $(b,y,\epsilon)\in\V(\Gamma)$, $\epsilon\in \ZZ_2$, by $(i,j)\in \ZZ_s\times \ZZ_t$. Now observe that a horizontal dart of 
$\Gamma$ points from a white vertex with label $(i,j)$ to a black vertex with label $(i+1,j)$ for some $(i,j)\in \ZZ_s\times\ZZ_t$, while a vertical dart points from a black vertex with label $(i,j)$ to a white vertex with label $(i,j+1)$. Therefore, a walk of length 2 
starting in a white vertex with label $(i,j)$
ends in a vertex with label $(i,j)$ if it consists of two vertical or of two horizontal darts (and the sum of the walk is then $0$), 
ends in a vertex with label $(i+1,j+1)$ if it first traverses a horizontal and then a vertical dart (and the sum of the walk is then $2$), and
ends in a vertex with label $(i-1,j-1)$ if it first traverses a vertical and then a horizontal dart (and the sum of the walk is then $-2$).
 From this, the following can be deduced immediately:

{\em A walk between two white vertices in $\Gamma$ has sum $2i$ for some integer $i$
and the labels of the two white vertices then differ by  $(i,i)\in\ZZ_s\times\ZZ_t$.}
(Here the element $(i,i)$ should be interpreted as $(i \hbox{ mod } s, i \hbox{ mod } t)$.)

From this, it follows that the connected component of $\Gamma$ that contains the vertex $(a,x,0)$ can contain only those white
vertices that have labels within the subgroup $\langle(1,1)\rangle$ generated by the element $(1,1) \in \ZZ_s\times\ZZ_t$.
On the other hand, since
every directed walk of length $2i$ starting in a white vertex with label $(0,0)$ will end in a white vertex with label 
$(i,i)$, it follows that:

{\em For every element $(i,i)\in \langle(1,1)\rangle$, there is 
at least one white vertex in the component of $(a,x,0)$ with label $(i,i)$}.

Let us now show that in fact all white vertices with label contained in $\langle(1,1)\rangle$ are in the same component as $(a,x,0)$.
Take a white vertex $(b,y,0)$ with label $(i,i) \in \langle(1,1)\rangle$ and let $(b',y',0)$ be the final vertex of a directed  walk of length
$2i$ starting in 
$(a,x,0)$.
 Observe that the label of $(b',y',0)$ is then $(i,i)$, implying that $b$ and $b'$ ($y$ and $y'$)  
are two vertices of $\Gamma_1$ ($\Gamma_2$)
 that belong to the same alter-class of $\Gamma_1$ ($\Gamma_2$, respectively).
 In particular, for some integer $k$, 
 there is a walk $W_1=b_0,b_1,\ldots,b_{2\ell k}$ of sum $0$ in $\Gamma_1$ with $b_0=b'$ and $b_{2\ell k}=b$,  
 and a walk $W_2=y_0,y_1,\ldots,y_{2\ell k}$ of sum $0$ in $\Gamma_2$ with $y_0=y'$ and $y_{2\ell k}=y$,
 both these walk being in the standard form and tolerance $k$.

We shall now construct a walk from $(a,x,0)$ to $(b,y,0)$ in $\Gamma$ by alternatingly using darts from $W_1$ and $W_2$;
that is, for $r\in \ZZ_\ell$ and $i\in\ZZ_k$, let
\begin{eqnarray*}
 U_{r,i} & = & (b_{2rk+i},y_{2rk+i},0)\,  (b_{2rk+i+1},y_{2rk+i},1),\\
  U_{r,i}^* & = & (b_{(2r+1)k+i},y_{(2r+1)k+1+i},1)\, (b_{(2r+1)k+1+i},y_{(2r+1)k+1+i},0),
\end{eqnarray*}
and observe that the concatenation
$$
W_r= U_{r,0}\, U_{r,1}\, \ldots \, U_{r,k-1}\,\,   U_{r,0}^*\, U_{r,1}^* \, \ldots \, U_{r,k-1}^*.
$$
is a walk in $\Gamma$ of sum $0$ from $(b_{2rk},y_{2rk},0)$ to $(b_{2(r+1)k},y_{2(r+1)k},0)$.
Finally, by concatenating the walks $W_r$ 
(where the end-vertex of $W_{r-1}$ is identified with the start-vertex of $W_r$ for every $r\in \{1,\ldots, \ell-1$),
one obtains a walk
$
W_0\, W_1\, \ldots, W_{\ell-1}
$
in $\Gamma$ of sum $0$ from $(b_{0},y_{0},0) = (b',y',0)$ to $(b_{2\ell k}, y_{2\ell k},0) = (b,y,0)$. 
%
%
In particular, $(b,y,0)$ is in the same connected component of $\Gamma$ as $(a,x,0)$.
We have thus shown that the set of white vertices that are in the same component of $\Gamma$ as a white vertex with label $(0,0)$
is precisely the set of white vertices with label within the group $\langle (1,1)\rangle \le \ZZ_s\times\ZZ_t$.

If we chose the labelling of the alter-classes of $\Gamma_1$ and $\Gamma_2$ in such a way that $a$ and $x$ had label $\alpha\in \ZZ_s$
and $\beta\in \ZZ_t$ (rather than $0$ and $0$), the label of a vertex $(b,y,0)$ that had label $(i,j)$ in the original labelling then receives
the label $(i+\alpha,j+\beta)$. By the result that we proved above, it follows that two white vertices belong to the same component of $\Gamma$ 
if and only if their label differ by an element of the group $\langle (1,1)\rangle \le \ZZ_s\times\ZZ_t$. This implies that the number of connected
component of $\Gamma$ equals the number of cosets of $\langle (1,1)\rangle$ in $\ZZ_s\times\ZZ_t$, which is clearly $\gcd(s,t)$.
\end{proof}

\section{Relationship with $\AAG(\Lambda)$}
\label{sec:AAG}

 Next, we wish to explain the motivating construction $\AAG(\Lambda)$, mentioned in Section~\ref{sec:intro}, in terms of the separated box product. 

  Let us first define the directed version of the $\AAG$ construction: For a graph $\Lambda$, let $\dAAG(\Lambda)$ be the digraph
  with vertex-set $\D(\Lambda)\times\D(\Lambda)$ and with a pair $\bigl( (x,y),(z,w) \bigr)$, $x,y,z,w\in \D(\Lambda)$, forming a dart of
  $\dAAG(\Lambda)$ if and only if $y=z$ and $(x,w)$ is a $2$-dart of $\Lambda$. Note that the graph $\AAG(\Lambda)$ is then
  just the underlying graph of $\dAAG(\Lambda)$.
 
  Following \cite{HC}, for a digraph $\Gamma$, we let the {\em dart digraph of $\Gamma$} (denoted $\Dart\Gamma$) be the
  the digraph with vertices and darts being the darts and $2$-darts of $\Gamma$, respectively. Note that the dart digraph is always
  reversible with a reversal which maps a dart $(u,v)$ to its inverse dart $(v,u)$.
  
  The {\em canonical double cover} of a digraph $\Gamma$ is the digraph $\CDC(\Gamma)$ with vertex-set being $\V(\Gamma) \times \ZZ_2$
  and dart-set $\{ ((u,i), (v,1-i)) : (u,v) \in \D(\Gamma), i\in \ZZ_2\}$. Note that for a connected digraph $\Gamma$,
  the digraph $\CDC(\Gamma)$ has two or one connected components, depending on whether
  $\Gamma$ is  bipartite or not. In the case $\Gamma$ is bipartite, it is isomorphic to each of the connected components of $\CDC(\Gamma)$.
  
Recall that $\Gamma_1\#\Gamma_2$ denotes a connected component of  the underlying graph of $\SBP(\Gamma_1,\Gamma_2)$.
  We can now state the result explaining the relationship between the $\AAG$-construction and the separated box product:
  
  \begin{lemma}
  \label{lem:iso}
  If $\Lambda$ is a connected cubic graph, then $\AAG(\Lambda) \cong \Dart\Lambda \# \Dart\Lambda$, and 
  moreover, $\CDC(\dAAG(\Lambda))\cong \SBP(\Dart\Lambda,\Dart\Lambda)$.
%
   %
  \end{lemma}
    
  \begin{proof}
  Consider the mapping $\varphi \colon \V(\CDC(\dAAG(\Lambda))) \to \V(\SBP(\Dart\Lambda,\Dart\Lambda))$ 
  defined by
  $$
  \varphi( (x,y), 0) = (x,y,0) \>\> \hbox{ and } \>\> \varphi( (x,y), 1) = (y,x,1)
  $$
  for every pair of darts $x,y \in \D(\Lambda)$.
  It is a matter of straightforward calculation to show that $\varphi$ is an isomorphism between 
  the two digraphs.
  
    Since $\dAAG(\Lambda)$ is connected, it is isomorphic to a connected component of its canonical double cover,
   and therefore isomorphic to a connected component of $\SBP(\Dart\Lambda,\Dart\Lambda)$;
   in particular,  $\AAG(\Lambda) \cong \U\dAAG(\Lambda) \cong \Dart\Lambda \# \Dart\Lambda$, as claimed.
\end{proof}

  Lemma~\ref{lem:iso} allows us to use the information about the symmetries of the separated box product
  to explain the unexpectedly large symmetry group of the graph $\AAG(\Lambda)$.

  Assume henceforth that $\Lambda$ is a connected cubic graph and that $G\le \Aut(\Lambda)$ acts transitively on 
  the $2$-darts of $\Lambda$. Let $\Delta = \Dart\Lambda$ and let $\Gamma=\SBP(\Delta,\Delta)$.
 Then $\Delta$ is a connected $2$-valent digraph with $G$ acting dart-transitively on $\Delta$ as a group of symmetries.
 
 Let $G\times G$ and $\tau$ act upon $\V(\Gamma)$ as in Corollary~\ref{the:delta2}; that is, let
  $$
   (x,y,i)^{(g,h)} = (x^g, y^h,i) \quad \hbox{and} \quad    (x,y,i)^\tau = (y, x,1-i)
  $$
  for every pair $x,y \in \D(\Lambda)$ and $i\in \ZZ_2$. Then, by Corollary~\ref{the:delta2},
  the group
   $$
    Y=\la G\times G, \tau \ra \cong G\wr C_2
   $$
   acts on $\Gamma$ as a dart-transitive group of symmetries.

  Note that $\Delta$ admits a reversal $\iota$ mapping a dart 
 $x=(u,v)$ of $\Lambda$ to its inverse dart $x^{-1}=(v,u)$.
 By Corollary~\ref{the:delta2}, the reversal $\iota$ of $\Delta$ gives rise to the
  reversal $\sigma$ of $\Gamma$ satisfying the formula
  $$
  (x,y,i)^\sigma = (x^{-1},y^{-1},1-i),
  $$
  and the group
  $$
   X=\la G\times G, \tau, \sigma \ra
  $$
  acts on the underlying graph $\U\Gamma$  dart-transitively.
  Since $\sigma$  commutes with $\tau$ and with each element of $G\times G$, the group 
  $X$ is isomorphic to $\langle G\times G, \tau\rangle \times \langle \sigma \rangle \cong (G\wr C_2) \times C_2$.
 Since, by Lemma~\ref{lem:iso}, $\CDC(\dAAG(\Lambda)) \cong \Gamma$, 
 the groups $Y$ and $X$ may be viewed as acting  dart-transitively on $\CDC(\dAAG(\Lambda))$ and $\CDC(\AAG(\Lambda))$,
 respectively.

 Let us now move our attention to the symmetries of $\dAAG(\Lambda)$ and $\AAG(\Lambda)$. 
 Observe first that dart-transitivity
 of the canonical double cover does not necessarily imply dart-tran\-si\-ti\-vi\-ty of the base graph,
 as, for example, observed in \cite{MarUnst} and further analysed in \cite{Surow,WilUnst}. This 
 phenomenon is caused by existence of {\em unexpected symmetries} of the canonical double cover
 that do not respect the partition of the vertices of the cover into fibres.
 

   As it happens, the digraph $\CDC(\dAAG(\Lambda))$, as well as its underlying graph $\CDC(\AAG(\Lambda))$,
   {\em always} possesses unexpected symmetries: 
    For $g\in \Aut(\Lambda)$, the mapping which maps a vertex $((x,y),0)$ of $\CDC(\dAAG(\Lambda))$ to
   $((x^g,y),0)$ and a vertex $((x,y),1)$ to $((x,y^g),1)$ is an unexpected symmetry of $\CDC(\dAAG(\Lambda))$.
Considering the isomorphism of Lemma~\ref{lem:iso}, this is equivalent to the symmetry$(g, I)$ of $\SBP(\Dart\Lambda,\Dart\Lambda)$, where $I$ stands for the trivial symmetry. 
  
  Due to these unexpected symmetries, one cannot conclude that $\dAAG(\Lambda)$ is dart-transitive.
   In fact,  in the case when $\dAAG(\Lambda)$ (and consequently $\Lambda$)
   is not bipartite (that is, when $\CDC(\dAAG(\Lambda))$ is connected), the digraph $\dAAG(\Lambda)$ and its
   underlying graph $\AAG(\Lambda)$ are indeed not dart-transitive.
   
   However, when $\Lambda$ (and consequently $\dAAG(\Lambda)$) is bipartite, then 
   $\CDC(\dAAG(\Lambda))$ decomposes into two connected components, say $\Omega$ and $\Omega'$,
   each isomorphic to $\dAAG(\Lambda)$.
   Then the set-wise stabiliser $Y_\Omega$ of $\Omega$ acts dart-transitively on $\Omega\cong \dAAG(\Lambda)$
   and the set-wise stabiliser $X_\Omega$ acts dart-transitively on the underlying graph $\U\Omega\cong \AAG(\Lambda)$.
    Hence the following result.
   
   \begin{lemma}
   If $\Lambda$ is a connected bipartite cubic $2$-dart-transitive graph, then
   $\dAAG(\Lambda)$ is a connected $2$-valent reversible dart-transitive digraph and its underlying graph
   $\AAG(\Lambda)$ is a connected tetravalent dart-transitive graph.
   \end{lemma}
   
   Let us analyse the groups $Y_\Omega$ and $X_{\Omega}$ in more detail.
   Assume for the rest of this section that $\Lambda$ is a connected bipartite cubic graph
   admitting a $2$-dart-transitive group of symmetries $G$,
   and let $\sigma$, $\tau$,  $Y$ and $X$ be as above. In view of the isomorphism given in the proof of Lemma~\ref{lem:iso},
   $G\times G$, $\sigma$ and $\tau$ may be viewed to act upon the vertices of $\CDC(\dAAG(\Lambda))$ according to:
   $$
   ((x,y),0)^{(g_1,g_2)} =    ((x^{g_1},y^{g_2}),0), \>\>
    ((x,y),1)^{(g_1,g_2)} =    ((x^{g_2},y^{g_1}),1) \> \hbox{ for }\> g_1,g_2\in G,
   $$
   $$
   ((x,y),i)^\sigma  = ((y^{-1},x^{-1}),1-i), \quad ((x,y),i)^\tau = ((x,y),1-i).
   $$
   
   Colour the vertices of $\Lambda$ properly black and white and let 
   a dart of $\Lambda$ inherit the colour of its initial vertex;
   this then determines a proper colouring of the vertices of $\Dart\Lambda$. Further, colour
   a vertex $(x,y)$ of $\dAAG(\Lambda)$ {\em blue} if the darts $x,y$ of $\Lambda$ are of the same colour (either black or white),
   and {\em red} otherwise;  this is then a proper colouring of the vertices of $\dAAG(\Lambda)$.
   Thus, $\CDC(\dAAG(\Lambda))$ has two connected components, one of them containing all the vertices
   $((x,y),0)$ with $(x,y)$ blue and all the vertices $((x,y),1)$ with $(x,y)$ red; let us denote this component by $\Omega$,
    and the other by $\Omega'$.
   Note that the mapping $\psi$ from $\dAAG(\Lambda)$ to $\Omega$ mapping a blue vertex $(x,y)$ to $((x,y),0)$ and
   a red vertex $(x,y)$ to $((x,y),1)$ gives rise to an isomorphism between these two digraphs. 
   
  Let $H$ be the index-$2$ subgroup of $G$ preserving the set of black vertices  $\Lambda$. 
  Observe that $H\times H$, in its action on $\CDC(\dAAG(\Lambda))$, preserves $\Omega$,
   and thus $H\times H \le Y_\Omega$.   
   On the other hand, the element $\tau$ maps $\Omega$ to $\Omega'$, implying that $\tau \in Y\setminus Y_\Omega$.
   
  Further, take $\alpha \in G\setminus H$, and observe that the elements $(\alpha,1)$ and $(1,\alpha)$ of $G\times G$, 
  in their action on $\CDC(\dAAG(\Lambda))$, switch $\Omega$ and $\Omega'$, implying that $(\alpha,\alpha), (1,\alpha)\tau$ and
  $(\alpha,1)\tau$ belong to $Y_\Omega$.
  
   Finally, observe that the reversal $\sigma$ interchanges $\Omega$ and $\Omega'$, implying that $\tau\sigma \in X_\Omega$.
   Now consider the groups
   $$
    A =   \langle H\times H, (\alpha,\alpha), (1,\alpha)\tau \rangle \quad \hbox{ and } \quad
    B =   \langle A, \tau\sigma \rangle,
   $$
   and observe that $\la A, \tau\ra = \la G\times G, \tau\ra = Y$.
   By the preceding discussion, we see that $A\le Y_\Omega$ and $B\le X_\Omega$. 
   Observe also that $H\times H$ is normal and has index $4$ in $A$,  implying that 
   $|A| = 4 |H\times H| = |G\times G| = |Y_\Omega|$, and showing that $A=Y_\Omega$. 
   Similarly, since $|X:Y| = |A:B|=2$, 
   implying that $B=X_\Omega$. Finally, since $|X:X_\Omega| =2$ we see that $X=\langle B,\tau\rangle$.
   Before proceeding,  let us prove the following elementary lemma.

   \begin{lemma}
   \label{lem:easy}
   Let $A$ be a normal subgroup of a group $X$ such that $X/A\cong \ZZ_2^2$, and let $b,d\in X$ be such that
   $X=\la A, b, d\ra$, $b^2, d^2, (bd)^2 \in A$. Suppose that $d$ is central in $X$. Then
   $\la A, b\ra \cong \la A, bd\ra$.
   \end{lemma}
   
   \begin{proof}
   Note that each element of $\la A, b\ra$ can be write uniquely in the form $ab^{i}$ for some $a\in A$ and $i\in \{0,1\}$.
   Let $\varphi \colon \la A, b\ra \to \la A, bd\ra$ be a mapping which maps such an element to $ab^id^i$.
   Then 
   $$
    \varphi(a_1b^{i_1} \cdot a_2b^{i_2}) = \varphi(a_1 a_2^{b^{-i_1}} b^{i_1+i_2}) = a_1 a_2^{b^{-i_1}} b^{i_1+i_2} d^{i_1+i_2} =
    a_1 b^{i_1}d^{i_1} a_2b^{i_2} d^{i_2},
   $$ 
   which clearly equals $\varphi(a_1b^{i_1}) \varphi(a_2b^{i_2})$,   
   showing that $\varphi$ is a group homomorphism, and in fact, isomorphism.
   \end{proof}
   
   Now, since $\tau$ and $\sigma$ are commuting involutions which normalise $A$, and since 
   $X=\la A, \tau, \sigma\ra$, we see that $X/A \cong C_2 \times C_2$. Since $\sigma$ is central in $X$,
   it now follows from Lemma~\ref{lem:easy} that $\la A, \tau\sigma \ra \cong \la A,\tau\ra$, and hence
   $B\cong Y \cong G\wr C_2$.
   

 The above discussion can be summarised as the following  Theorem.
   
   \begin{theorem}
   \label{the:A2G}
   Let $\Lambda$ be a connected cubic bipartite graph admitting a $2$-dart-transitive group of symmetries $G$.
   With the notation as above, it follows that
   $\dAAG(\Lambda)$ admits a dart-transitive group of symmetries isomorphic to $A$, and
   the underlying graph $\AAG(\Lambda)$ admits a dart-transitive group of symmetries isomorphic to
   $B$, which itself if isomorphic to $G \wr C_2$.
   \end{theorem}

\section{Application to tetravalent edge-transitive graphs}
\label{sec:4}
 
 While  the separated box product is a construction that applies to digraphs of all valences, we are most interested in its contributions to the {\em Census of edge-transitive teravalent graphs} \cite{C4}, a long-term project the aim of which is to compile an extensive, 
 possibly exhaustive list of small tetravalent edge-transitive graphs, and explain the way in which such graphs arise.

  It is well known that tetravalent edge-transitive graphs fall into three classes:
 
 First class consists of {\em dart-transitive} graphs, that is those $\Gamma$, where $\Aut(\Gamma)$ acts transitively on $\D(\Gamma)$. 
 These are the most widely studied class of tetravalent edge-transitive graphs; and due to a recent progress in \cite{PSV},
 all such graphs of order at most $640$ are now known \cite{PSVcubic} and included in the on-line census \cite{C4}.
 For convenience, we say that dart-transitive graphs have {\em symmetry type DT}.
 
 The second class consists of {\em half-transitive} graphs $\Gamma$ (also known in the literature as {\em half-arc-transitive} graphs),
 in which $\Aut(\Gamma)$ acts transitively on the edges, vertices, but not on the darts of $\Gamma$.
 This family of graphs was first considered by Tutte \cite{tutte}, who observed that each such graph is
  the underlying graphs of a dart-transitive $2$-valent digraph. Later, these graph received
  considerable attention; see \cite{Bouwer,CPS,ConZit,Marusic1998a,MarusicNedela2001a,Cheryl,Spa08}, for example. 
  Using results from \cite{SV}
  all tetravalent half-transitive graphs (and what is more, all $2$-valent dart-transitive digraphs) 
  of order at most $1000$ are now known \cite{PSV2}.
  We will say that half-transitive graphs have {\em symmetry type HT}.

  The third class consists of the {\em semisymmetric} graphs $\Gamma$, where $\Aut(\Gamma)$ 
 is transitive on the edges, but neither on darts nor on the vertices of $\Gamma$;
 the first example of such a graph was found by Folkman~\cite{folk}.
 It is well known and easy to show that these graphs are necessarily bipartite with
  each part of the bipartition forming a vertex-orbit of the symmetry group.
 Unlike the dart-transitive and half-transitive counterparts, the tetravalent semisymmetric graphs
 seem to be harder to enumerate. No exhaustive list of such graphs of small order (not even up to order $100$, say),
  exists so far.
 We also say that semisymmetric graphs have {\em symmetry type SS}.
 
  As was shown in \cite{girth4}, the family of tetravalent semisymmetric graphs is closely related to another class of vertex-transitive
  tetravalent graphs, the class of {\em linking ring structures}, which we now define: 
  
  Let $\Gamma$ be a tetravalent graph and let $G\le \Aut(\Gamma)$. If $G$ is transitive on the vertices, has
  two orbits on the edges as well as on the darts of $\Gamma$ and if the stabiliser of any dart of $\Gamma$ in $G$
  is non-trivial, then $\Gamma$ is called a {\em $G$-linking-rings structure}, or $G$-LR, for short. (Note that
  this implies that $G_v^{\Gamma(v)}$ is the Klein $4$-group in its intransitive action on $4$ points.)
  If $\Gamma$ is
  $\Aut(\Gamma)$-LR, then we say that it is of {\em symmetry type LR}.
   Note that the symmetry type of a $G$-LR graph is LR or DT.
   
  Let us mention that  each $G$-LR yields
  a bipartite edge-transitive tetravalent graph of girth $4$ via a the {\em partial line graph construction} (see \cite{girth4} for details),
  and if the $G$-LR is suitable, this tetravalent graph is of type SS (and of type DT otherwise).
  Conversely, every tetravalent semisymmetric graph of girth $4$ is contained in this way. For more information on LR graphs,
  see \cite{LR1,LR2}.

 

In order to construct a tetravalent graph of one of the above mentioned symmetry types, one can try to take
two $2$-valent dart-transitive digraphs $\Gamma_1$ and $\Gamma_2$ and apply the $\SBP$-construction to them.
Corollaries~\ref{cor:LR}, \ref{the:delta2}, and \ref{the:SS}, then assure certain symmetries of 
the underlying graph $\U\SBP(\Gamma_1,\Gamma_2)$;
we shall refer to the symmetries predicted by these three corollaries as {\em expected symmetries} and
the symmetry type that  $\U\SBP(\Gamma_1,\Gamma_2)$ would have if it only had expected symmetries will
be called the {\em expected symmetry type}.

\begin{theorem}
\label{the:type}
For $i\in \{1,2\}$, let $\Gamma_i$ be a connected $G_i$-dart-transitve $2$-valent digraph,
Let $\Gamma=\SBP(\Gamma_1,\Gamma_2)$ and let $\U\Gamma$ be the underlying graph of $\Gamma$.
Then 
the following holds:
\begin{itemize}
\item[{\rm (i)}]
If
 $\Gamma_1\not \cong \Gamma_2$, $\Gamma_1\not \cong \Gamma_2^{-1}$ and
both $\Gamma_1$ and $\Gamma_2$ are reversible,
then the expected symmetry type of $\U\Gamma$ is LR.
\item[{\rm (ii)}]
If 
$\Gamma_1\cong \Gamma_2$ and $\Gamma_1\not \cong \Gamma_2^{-1}$ (in particular, $\Gamma_1$ and $\Gamma_2$ are not reversible),
then the expected symmetry type of $\U\Gamma$ is HT.
\item[{\rm (iii)}]
If 
$\Gamma_1\not \cong \Gamma_2$ and $\Gamma_1\cong \Gamma_2^{-1}$ (in particular, $\Gamma_1$ and $\Gamma_2$ are not reversible),
then the expected symmetry type of $\U\Gamma$ is SS.
\item[{\rm (iv)}]
If 
$\Gamma_1\cong \Gamma_2$ and $\Gamma_1\cong \Gamma_2^{-1}$ (in particular, $\Gamma_1$ and $\Gamma_2$ are reversible),
then the expected symmetry type of $\U\Gamma$ is DT.
\end{itemize}
\end{theorem}

\begin{proof}
To prove part (i), observe first that the existence of a group $G\le \Aut(\U\Gamma)$ such that $\U\Gamma$ is a $G$-LR 
follows directly from Corollary~\ref{cor:LR}. The symmetry type of $\U\Gamma$ is thus either LR or DT. However, since $\Gamma_1$ is
not isomorphic neither to $\Gamma_2^{-1}$ nor to $\Gamma_2$, neither Corollary~\ref{the:SS} nor Corollary~\ref{the:delta2}
 applies and hence no
other expected symmetries exist. The expected symmetry type of $\U\Gamma$ is thus LR.

To prove part (ii), note that since $\Gamma_1\cong\Gamma_2\not\cong\Gamma_2^{-1}$, 
the digraph $\Gamma_1$ possesses no reversal, 
and that neither of Corollary~\ref{cor:LR}, Corollary~\ref{the:SS} or the second
paragraph of Corollary~\ref{the:delta2} applies. By the first paragraph of Corollary~\ref{the:delta2},
it therefore follows that the expected symmetry type of $\U\Gamma$ is HT.

The proof of part (iii) is similar, except that here only Corollary~\ref{the:SS} applies.

Finally, in view of the Remark below Corollary~\ref{the:SS}, in the case of (iv), the graph $\U\Gamma$ is dart-transitive.
\end{proof}

\subsection{Four Tables}

We conclude with four tables, each one displaying prducts of each of the types mentioned in Theorem \ref{the:type}.   The orientations used for input in these products are of two kinds.  First, the site  \cite{PotWeb} (see also \cite{PSV2}) catalogs all  $2$-valent dart-transitive orientations on up to $1\,000$ vertices, and ``ATD$[n, i]$" refers to the $i$-th orientation among those of order $n$ in that list. Second, the notation ``$\DCyc[n]$'' means the graph $C_n$ considered as a digraph of in- and out-valence 2.

Other headings in the tables: ``vs'' stands for the order of a vertex-stabiliser, and  ``AP'' indicates the alter-perimeter of the orientation.  Since one is mainly interested in connected graphs,  we include properties of the connected component $\Gamma_1\#\Gamma_2$ of  $\SBP(\Gamma_1,\Gamma_2)$

Table~\ref{T1} considers products of non-isomorphic reversible digraphs.   Theorem~\ref{the:type} tells us the expected type is LR, though many of these are in fact DT. Table~\ref{T2} shows products of a reversible digraph with itself.  As the theorem shows, these are all DT. Table~\ref{T3} shows products of a non-reversible digraph with itself.  We expect these to be HT and they all are.   Table~\ref{T4} shows products of a non-reversible digraph with its reverse.  We expect these to be SS and they all are.

\vfill

\begin{center}
\begin{table}[H]
\begin{tiny}
\bigskip
\begin{tabular}{||c|c|c||c|c|c||c|c|c|c|c|c||}
\hline\hline\multicolumn{3}{||c||}{$\Delta_1$} &
\multicolumn{3}{c||}{$\Delta_2$} &
\multicolumn{6}{c||}{$\Gamma = \Delta_1\#\Delta_2$} \\
\hline \hline

\hline \hline
 Name &V & AP  & Name &V  & AP  & V  &AP& vs & girth & diam & SymType  
\\
  \hline\hline
ATD[6, 1] & 6 & 3 & DCyc[3] & 3  & 1 & 36 & 6 & $8$ & 4 & 5 & LR  \\  \hline
ATD[6, 1] & 6 & 3 & DCyc[4] & 4  & 2 & 48 & 12 & $2^{11}$ & 4 & 6 & DT  \\  
\hline
ATD[6, 1] & 6 & 3 & DCyc[5] & 5  & 1 & 60 & 6 & $8$ & 4 & 6 & LR  \\  \hline
ATD[6, 1] & 6 & 3 & DCyc[6] & 6  & 2 & 72 & 12 & $64$ & 4 & 6 & LR  \\  \hline
ATD[6, 1] & 6 & 3 & DCyc[7] & 7  & 1 & 84 & 6 & $8$ & 4 & 6 & LR  \\  \hline
ATD[6, 1] & 6 & 3 & ATD[8, 1] & 8  & 2 & 96 & 12 & $64$ & 4 & 7 & LR  \\  
\hline
ATD[6, 1] & 6 & 3 & ATD[8, 2] & 8  & 4 & 96 & 24 & $2^{23}$ & 4 & 12 & DT  
\\  \hline
ATD[6, 1] & 6 & 3 & DCyc[8] & 8  & 2 & 96 & 12 & $64$ & 4 & 7 & LR  \\  \hline
ATD[6, 1] & 6 & 3 & ATD[12, 4] & 12  & 3 & 48 & 6 & $64$ & 4 & 6 & LR  \\  
\hline
ATD[6, 1] & 6 & 3 & ATD[18, 3] & 18  & 9 & 72 & 18 & $2^{17}$ & 4 & 9 & DT  
\\  \hline
ATD[6, 1] & 6 & 3 & ATD[24, 7] & 24  & 3 & 96 & 6 & $64$ & 4 & 6 & LR  \\  
\hline
ATD[6, 1] & 6 & 3 & ATD[24, 9] & 24  & 3 & 96 & 6 & $64$ & 4 & 6 & LR  \\  
\hline
ATD[6, 1] & 6 & 3 & ATD[24, 14] & 24  & 6 & 96 & 12 & $2^{15}$ & 4 & 6 & LR 
\\  \hline
ATD[8, 1] & 8 & 2 & DCyc[3] & 3  & 1 & 48 & 4 & $4$ & 6 & 5 & LR  \\  \hline
ATD[8, 1] & 8 & 2 & DCyc[4] & 4  & 2 & 32 & 4 & $4$ & 4 & 4 & LR  \\  \hline
ATD[8, 1] & 8 & 2 & DCyc[5] & 5  & 1 & 80 & 4 & $4$ & 6 & 7 & LR  \\  \hline
ATD[8, 1] & 8 & 2 & ATD[8, 2] & 8  & 4 & 64 & 8 & $16$ & 4 & 6 & LR  \\  
\hline
ATD[8, 1] & 8 & 2 & DCyc[8] & 8  & 2 & 64 & 4 & $4$ & 6 & 6 & LR  \\  \hline
ATD[8, 1] & 8 & 2 & ATD[10, 1] & 10  & 2 & 80 & 4 & $4$ & 6 & 5 & LR  \\  
\hline
ATD[8, 1] & 8 & 2 & ATD[12, 2] & 12  & 4 & 96 & 8 & $4$ & 6 & 7 & LR  \\  
\hline
ATD[8, 1] & 8 & 2 & ATD[12, 3] & 12  & 2 & 96 & 4 & $4$ & 6 & 6 & LR  \\  
\hline
ATD[8, 1] & 8 & 2 & DCyc[12] & 12  & 2 & 96 & 4 & $4$ & 6 & 8 & LR  \\  \hline
ATD[8, 2] & 8 & 4 & DCyc[3] & 3  & 1 & 48 & 8 & $16$ & 4 & 5 & LR  \\  \hline
ATD[8, 2] & 8 & 4 & DCyc[4] & 4  & 2 & 32 & 8 & $2^{7}$ & 4 & 4 & DT  \\  
\hline
ATD[8, 2] & 8 & 4 & DCyc[5] & 5  & 1 & 80 & 8 & $16$ & 4 & 7 & LR  \\  \hline
ATD[8, 2] & 8 & 4 & ATD[10, 1] & 10  & 2 & 80 & 8 & $16$ & 4 & 6 & LR  \\  
\hline
ATD[8, 2] & 8 & 4 & ATD[12, 3] & 12  & 2 & 96 & 8 & $16$ & 4 & 8 & LR  \\  
\hline
ATD[8, 2] & 8 & 4 & ATD[16, 2] & 16  & 4 & 64 & 8 & $16$ & 4 & 6 & LR  \\  
\hline
ATD[8, 2] & 8 & 4 & ATD[16, 4] & 16  & 4 & 64 & 8 & $2^{9}$ & 4 & 6 & LR  \\
\hline
ATD[8, 2] & 8 & 4 & ATD[16, 5] & 16  & 8 & 64 & 16 & $2^{15}$ & 4 & 8 & DT  
\\  \hline
ATD[8, 2] & 8 & 4 & ATD[24, 3] & 24  & 8 & 96 & 16 & $2^{8}$ & 4 & 8 & LR  
\\  \hline
ATD[8, 2] & 8 & 4 & ATD[24, 5] & 24  & 4 & 96 & 8 & $16$ & 4 & 8 & LR  \\  
\hline
ATD[9, 1] & 9 & 3 & DCyc[3] & 3  & 1 & 54 & 6 & $24$ & 6 & 5 & DT  \\  \hline
ATD[9, 1] & 9 & 3 & DCyc[5] & 5  & 1 & 90 & 6 & $4$ & 6 & 6 & LR  \\  \hline
ATD[9, 1] & 9 & 3 & ATD[12, 4] & 12  & 3 & 72 & 6 & $4$ & 4 & 6 & LR  \\  
\hline
ATD[10, 1] & 10 & 2 & DCyc[3] & 3  & 1 & 60 & 4 & $4$ & 6 & 6 & LR  \\  \hline
ATD[10, 1] & 10 & 2 & DCyc[4] & 4  & 2 & 40 & 4 & $4$ & 4 & 5 & LR  \\  \hline
ATD[10, 1] & 10 & 2 & DCyc[5] & 5  & 1 & 100 & 4 & $4$ & 6 & 8 & LR  \\  
\hline
ATD[10, 1] & 10 & 2 & DCyc[8] & 8  & 2 & 80 & 4 & $4$ & 6 & 7 & LR  \\  \hline
ATD[10, 2] & 10 & 5 & DCyc[3] & 3  & 1 & 60 & 10 & $32$ & 4 & 5 & LR  \\  
\hline
ATD[10, 2] & 10 & 5 & DCyc[4] & 4  & 2 & 80 & 20 & $2^{19}$ & 4 & 10 & DT  \\ 
\hline
ATD[10, 2] & 10 & 5 & DCyc[5] & 5  & 1 & 100 & 10 & $32$ & 4 & 7 & LR  \\  
\hline
ATD[10, 2] & 10 & 5 & ATD[20, 4] & 20  & 5 & 80 & 10 & $2^{12}$ & 4 & 6 & LR
\\  \hline
ATD[12, 1] & 12 & 1 & DCyc[3] & 3  & 1 & 72 & 2 & $4$ & 6 & 6 & LR  \\  \hline
ATD[12, 1] & 12 & 1 & DCyc[4] & 4  & 2 & 96 & 4 & $2^{8}$ & 4 & 7 & LR  \\  
\hline
ATD[12, 2] & 12 & 4 & DCyc[3] & 3  & 1 & 72 & 8 & $8$ & 6 & 6 & DT  \\  \hline
ATD[12, 2] & 12 & 4 & ATD[16, 2] & 16  & 4 & 96 & 8 & $4$ & 6 & 7 & LR  \\  
\hline
ATD[12, 2] & 12 & 4 & ATD[16, 4] & 16  & 4 & 96 & 8 & $8$ & 4 & 7 & LR  \\  \hline
ATD[12, 3] & 12 & 2 & DCyc[3] & 3  & 1 & 72 & 4 & $4$ & 6 & 6 & LR  \\  \hline
ATD[12, 3] & 12 & 2 & DCyc[4] & 4  & 2 & 48 & 4 & $4$ & 4 & 5 & LR  \\  \hline
ATD[12, 3] & 12 & 2 & DCyc[8] & 8  & 2 & 96 & 4 & $4$ & 6 & 7 & LR  \\  \hline
ATD[14, 1] & 14 & 7 & DCyc[3] & 3  & 1 & 84 & 14 & $2^{7}$ & 4 & 7 & LR  \\  \hline
ATD[15, 2] & 15 & 5 & DCyc[3] & 3  & 1 & 90 & 10 & $8$ & 6 & 6 & DT  \\  \hline
ATD[16, 3] & 16 & 2 & DCyc[3] & 3  & 1 & 96 & 4 & $4$ & 6 & 7 & LR  \\  \hline
ATD[16, 3] & 16 & 2 & DCyc[4] & 4  & 2 & 64 & 4 & $4$ & 4 & 6 & LR  \\  \hline
ATD[18, 2] & 18 & 2 & DCyc[4] & 4  & 2 & 72 & 4 & $2^{7}$ & 4 & 5 & LR  \\  \hline
ATD[20, 3] & 20 & 2 & DCyc[4] & 4  & 2 & 80 & 4 & $4$ & 4 & 7 & LR  \\  \hline
ATD[24, 2] & 24 & 2 & DCyc[4] & 4  & 2 & 96 & 4 & $4$ & 4 & 6 & LR  \\  \hline
ATD[24, 6] & 24 & 2 & DCyc[4] & 4  & 2 & 96 & 4 & $4$ & 4 & 6 & LR  \\  \hline
ATD[24, 13] & 24 & 2 & DCyc[4] & 4  & 2 & 96 & 4 & $4$ & 4 & 8 & LR  \\  
\hline
\end{tabular}
\vspace{2mm}
\end{tiny}
\caption{Small examples of  $\Delta_1\#\Delta_2$ with $\Delta_1\not\cong \Delta_2$, both reversible.}
\label{T1}
\end{table}
\end{center}

\begin{center}
\begin{table}
\begin{tiny}
\bigskip
\begin{tabular}{||c|c|c||c|c|c|c|c|c||}
\hline\hline
\multicolumn{3}{||c||}{$\Delta$} &
\multicolumn{6}{c||}{$\Gamma = \Delta\#\Delta$} \\
\hline \hline
 Name &V &  AP  &  V & $AP$ & vs & girth & diam & SymType  \\
  \hline\hline
ATD[6, 1] & 6 &  3 & 24 &  6  & $32$  & 4 & 4 & DT  \\  \hline
ATD[8, 1] & 8 &  2 & 64 &  4  & $8$  & 6 & 5 & DT  \\  \hline
ATD[8, 2] & 8 &  4 & 32 &  8  & $2^{7}$  & 4 & 4 & DT  \\  \hline
ATD[9, 1] & 9 &  3 & 54 &  6  & $24$  & 6 & 5 & DT  \\  \hline
ATD[10, 1] & 10 &  2 & 100 &  4  & $8$  & 8 & 6 & DT  \\  \hline
ATD[10, 2] & 10 &  5 & 40 &  10  & $2^{9}$  & 4 & 5 & DT  \\  \hline
ATD[12, 1] & 12 &  1 & 288 &  2  & $8$  & 6 & 9 & DT  \\  \hline
ATD[12, 2] & 12 &  4 & 72 &  8  & $8$  & 6 & 6 & DT  \\  \hline
ATD[12, 3] & 12 &  2 & 144 &  4  & $8$  & 6 & 7 & DT  \\  \hline
ATD[12, 4] & 12 &  3 & 96 &  6  & $8$  & 4 & 6 & DT  \\  \hline
ATD[12, 5] & 12 &  6 & 48 &  12  & $2^{11}$  & 4 & 6 & DT  \\  \hline
ATD[14, 1] & 14 &  7 & 56 &  14  & $2^{13}$  & 4 & 7 & DT  \\  \hline
ATD[15, 1] & 15 &  3 & 150 &  6  & $8$  & 6 & 7 & DT  \\  \hline;
ATD[15, 2] & 15 &  5 & 90 &  10  & $8$  & 6 & 6 & DT  \\  \hline
ATD[16, 1] & 16 &  4 & 128 &  8  & $8$  & 6 & 8 & DT  \\  \hline
ATD[16, 2] & 16 &  4 & 128 &  8  & $24$  & 6 & 8 & DT  \\  \hline
ATD[16, 3] & 16 &  2 & 256 &  4  & $8$  & 6 & 9 & DT  \\  \hline
ATD[16, 4] & 16 &  4 & 128 &  8  & $32$  & 4 & 8 & DT  \\  \hline
ATD[16, 5] & 16 &  8 & 64 &  16  & $2^{15}$  & 4 & 8 & DT  \\  \hline
ATD[18, 1] & 18 &  6 & 108 &  12  & $8$  & 6 & 6 & DT  \\  \hline
ATD[18, 2] & 18 &  2 & 324 &  4  & $32$  & 6 & 7 & DT  \\  \hline
ATD[18, 3] & 18 &  9 & 72 &  18  & $2^{17}$  & 4 & 9 & DT  \\  \hline
ATD[20, 1] & 20 &  4 & 200 &  8  & $8$  & 8 & 8 & DT  \\  \hline
ATD[20, 2] & 20 &  4 & 200 &  8  & $8$  & 6 & 10 & DT  \\  \hline
ATD[20, 4] & 20 &  5 & 160 &  10  & $2^{7}$  & 4 & 8 & DT  \\  \hline
ATD[20, 5] & 20 &  10 & 80 &  20  & $2^{19}$  & 4 & 10 & DT  \\  \hline
ATD[21, 3] & 21 &  3 & 294 &  6  & $8$  & 6 & 9 & DT  \\  \hline
ATD[21, 4] & 21 &  7 & 126 &  14  & $8$  & 6 & 7 & DT  \\  \hline
ATD[22, 1] & 22 &  11 & 88 &  22  & $2^{21}$  & 4 & 11 & DT  \\  \hline
ATD[24, 1] & 24 &  6 & 192 &  12  & $8$  & 6 & 8 & DT  \\  \hline
ATD[24, 3] & 24 &  8 & 144 &  16  & $8$  & 6 & 8 & DT  \\  \hline
ATD[24, 4] & 24 &  6 & 192 &  12  & $8$  & 6 & 8 & DT  \\  \hline
ATD[24, 5] & 24 &  4 & 288 &  8  & $8$  & 6 & 8 & DT  \\  \hline
ATD[24, 12] & 24 &  4 & 288 &  8  & $8$  & 6 & 12 & DT  \\  \hline
ATD[24, 14] & 24 &  6 & 192 &  12  & $2^{9}$  & 4 & 8 & DT  \\  \hline
ATD[24, 15] & 24 &  12 & 96 &  24  & $2^{23}$  & 4 & 12 & DT  \\  \hline
ATD[25, 1] & 25 &  5 & 250 &  10  & $24$  & 6 & 9 & DT  \\  \hline
ATD[26, 2] & 26 &  13 & 104 &  26  & $2^{25}$  & 4 & 13 & DT  \\  \hline
ATD[27, 3] & 27 &  9 & 162 &  18  & $8$  & 6 & 9 & DT  \\  \hline
ATD[28, 3] & 28 &  7 & 224 &  14  & $2^{11}$  & 4 & 8 & DT  \\  \hline
ATD[28, 4] & 28 &  14 & 112 &  28  & $2^{27}$  & 4 & 14 & DT  \\  \hline
ATD[30, 1] & 30 &  6 & 300 &  12  & $8$  & 6 & 10 & DT  \\  \hline
ATD[30, 3] & 30 &  10 & 180 &  20  & $8$  & 6 & 10 & DT  \\  \hline
ATD[30, 4] & 30 &  6 & 300 &  12  & $8$  & 8 & 8 & DT  \\  \hline
ATD[30, 7] & 30 &  15 & 120 &  30  & $2^{29}$  & 4 & 15 & DT  \\  \hline
ATD[32, 3] & 32 &  8 & 256 &  16  & $8$  & 6 & 8 & DT  \\  \hline
ATD[32, 5] & 32 &  8 & 256 &  16  & $8$  & 6 & 8 & DT  \\  \hline
ATD[32, 12] & 32 &  8 & 256 &  16  & $2^{13}$  & 4 & 8 & DT  \\  \hline
ATD[32, 13] & 32 &  16 & 128 &  32  & $2^{31}$  & 4 & 16 & DT  \\  \hline
ATD[33, 2] & 33 &  11 & 198 &  22  & $8$  & 6 & 11 & DT  \\  \hline
ATD[34, 2] & 34 &  17 & 136 &  34  & $2^{33}$  & 4 & 17 & DT  \\  \hline
ATD[36, 8] & 36 &  12 & 216 &  24  & $8$  & 6 & 12 & DT  \\  \hline
ATD[36, 13] & 36 &  9 & 288 &  18  & $2^{15}$  & 4 & 9 & DT  \\  \hline
ATD[36, 14] & 36 &  18 & 144 &  36  & $2^{35}$  & 4 & 18 & DT  \\  \hline
ATD[38, 1] & 38 &  19 & 152 &  38  & $2^{37}$  & 4 & 19 & DT  \\  \hline
ATD[39, 4] & 39 &  13 & 234 &  26  & $8$  & 6 & 13 & DT  \\  \hline
ATD[40, 4] & 40 &  10 & 320 &  20  & $8$  & 6 & 10 & DT  \\  \hline
ATD[40, 7] & 40 &  10 & 320 &  20  & $8$  & 6 & 10 & DT  \\  \hline
ATD[40, 13] & 40 &  10 & 320 &  20  & $2^{17}$  & 4 & 10 & DT  \\  \hline
ATD[40, 14] & 40 &  20 & 160 &  40  & $2^{39}$  & 4 & 20 & DT  \\  \hline
ATD[42, 4] & 42 &  14 & 252 &  28  & $8$  & 6 & 14 & DT  \\  \hline
ATD[42, 7] & 42 &  21 & 168 &  42  & $2^{41}$  & 4 & 21 & DT  \\  \hline
ATD[44, 4] & 44 &  22 & 176 &  44  & $2^{43}$  & 4 & 22 & DT  \\  \hline
ATD[45, 3] & 45 &  15 & 270 &  30  & $8$  & 6 & 15 & DT  \\  \hline
ATD[46, 1] & 46 &  23 & 184 &  46  & $2^{45}$  & 4 & 23 & DT  \\  \hline
ATD[48, 4] & 48 &  16 & 288 &  32  & $8$  & 6 & 16 & DT  \\  \hline
ATD[48, 31] & 48 &  24 & 192 &  48  & $2^{47}$  & 4 & 24 & DT  \\  \hline
ATD[50, 5] & 50 &  25 & 200 &  50  & $2^{49}$  & 4 & 25 & DT  \\  \hline
ATD[51, 2] & 51 &  17 & 306 &  34  & $8$  & 6 & 17 & DT  \\  \hline
ATD[52, 5] & 52 &  26 & 208 &  52  & $2^{51}$  & 4 & 26 & DT  \\  \hline
ATD[54, 7] & 54 &  18 & 324 &  36  & $8$  & 6 & 18 & DT  \\  \hline
ATD[54, 10] & 54 &  27 & 216 &  54  & $2^{53}$  & 4 & 27 & DT  \\  \hline
ATD[56, 11] & 56 &  28 & 224 &  56  & $2^{55}$  & 4 & 28 & DT  \\  \hline
ATD[58, 2] & 58 &  29 & 232 &  58  & $2^{57}$  & 4 & 29 & DT  \\  \hline
ATD[60, 27] & 60 &  30 & 240 &  60  & $2^{59}$  & 4 & 30 & DT  \\  \hline
\hline
\end{tabular}
\vspace{2mm}
\end{tiny}
\caption{Small examples of $\Delta\#\Delta$ for reversible $\Delta$.}
\label{T2}
\end{table}
\end{center}

\begin{center}
\begin{table}
\begin{tiny}
\bigskip
\begin{tabular}{||c|c|c||c|c|c|c|c|c||}
\hline\hline
\multicolumn{3}{||c||}{$\Delta$} &
\multicolumn{6}{c||}{$\Gamma = \Delta\#\Delta$} \\
\hline \hline
 Name &V &  AP  &  V & AP & vs & girth & diam & SymType  \\
  \hline\hline
ATD[21, 1] & 21 &  3 & 294 &  6  & $4$  & 6 & 8 & HT  \\  \hline
ATD[27, 1] & 27 &  3 & 486 &  6  & $4$  & 8 & 8 & HT  \\  \hline
ATD[39, 1] & 39 &  3 & 1014 &  6  & $4$  & 6 & 10 & HT  \\  \hline
ATD[42, 1] & 42 &  6 & 588 &  12  & $4$  & 8 & 8 & HT  \\  \hline
ATD[54, 1] & 54 &  6 & 972 &  12  & $4$  & 8 & 10 & HT  \\  \hline
ATD[55, 1] & 55 &  5 & 1210 &  10  & $4$  & 8 & 9 & HT  \\  \hline
ATD[55, 3] & 55 &  5 & 1210 &  10  & $4$  & 8 & 9 & HT  \\  \hline
ATD[57, 1] & 57 &  3 & 2166 &  6  & $4$  & 6 & 12 & HT  \\  \hline
ATD[60, 1] & 60 &  4 & 1800 &  8  & $4$  & 8 & 10 & HT  \\  \hline
ATD[63, 1] & 63 &  9 & 882 &  18  & $4$  & 8 & 9 & HT  \\  \hline
ATD[63, 3] & 63 &  3 & 2646 &  6  & $4$  & 6 & 12 & HT  \\  \hline
ATD[68, 1] & 68 &  4 & 2312 &  8  & $4$  & 8 & 10 & HT  \\  \hline
ATD[72, 1] & 72 &  2 & 5184 &  4  & $4$  & 8 & 11 & HT  \\  \hline
ATD[78, 1] & 78 &  6 & 2028 &  12  & $4$  & 8 & 10 & HT  \\  \hline
ATD[78, 3] & 78 &  6 & 2028 &  12  & $4$  & 8 & 10 & HT  \\  \hline
ATD[80, 1] & 80 &  4 & 3200 &  8  & $4$  & 8 & 12 & HT  \\  \hline
ATD[80, 3] & 80 &  4 & 3200 &  8  & $4$  & 8 & 12 & HT  \\  \hline
ATD[81, 1] & 81 &  9 & 1458 &  18  & $4$  & 8 & 10 & HT  \\  \hline
ATD[81, 3] & 81 &  3 & 4374 &  6  & $4$  & 8 & 13 & HT  \\  \hline
ATD[84, 1] & 84 &  3 & 4704 &  6  & $4$  & 6 & 14 & HT  \\  \hline
ATD[84, 3] & 84 &  12 & 1176 &  24  & $4$  & 8 & 12 & HT  \\  \hline
ATD[84, 5] & 84 &  6 & 2352 &  12  & $4$  & 8 & 12 & HT  \\  \hline
ATD[84, 18] & 84 &  6 & 2352 &  12  & $4$  & 8 & 12 & HT  \\  \hline
ATD[93, 1] & 93 &  3 & 5766 &  6  & $4$  & 6 & 15 & HT  \\  \hline
ATD[100, 1] & 100 &  4 & 5000 &  8  & $4$  & 8 & 12 & HT  \\  \hline
\hline
\end{tabular}
\vspace{2mm}
\end{tiny}
\caption{Small examples of $\Delta\#\Delta$ for non-reversible $\Delta$.}
\label{T3}
\end{table}
\end{center}

\begin{center}
\begin{table}
\begin{tiny}
\bigskip
\begin{tabular}{||c|c|c||c|c||c|c|c|c||}
\hline\hline
\multicolumn{3}{||c||}{$\Delta$} &
\multicolumn{6}{c||}{$\Gamma = \Delta\#\Delta^{-1}$} \\
\hline \hline
 Name &V &  AP  &  V & AP & vs & girth & diam & SymType  \\
  \hline\hline
ATD[ 21, 1 ] & 21 &  3 & 294 &  6  & $8$  & 6 & 8 & SS  \\  \hline
ATD[27, 1] & 27 &  3 & 486 &  6  & $8$  & 8 & 8 & SS  \\  \hline
ATD[39, 1] & 39 &  3 & 1014 &  6  & $8$  & 6 & 10 & SS  \\  \hline
ATD[42, 1] & 42 &  6 & 588 &  12  & $8$  & 8 & 8 & SS  \\  \hline
ATD[54, 1] & 54 &  6 & 972 &  12  & $8$  & 8 & 10 & SS  \\  \hline
ATD[55, 1] & 55 &  5 & 1210 &  10  & $8$  & 8 & 9 & SS  \\  \hline
ATD[55, 3] & 55 &  5 & 1210 &  10  & $8$  & 8 & 9 & SS  \\  \hline
ATD[57, 1] & 57 &  3 & 2166 &  6  & $8$  & 6 & 12 & SS  \\  \hline
ATD[60, 1] & 60 &  4 & 1800 &  8  & $8$  & 8 & 10 & SS  \\  \hline
ATD[63, 1] & 63 &  9 & 882 &  18  & $8$  & 8 & 9 & SS  \\  \hline
ATD[63, 3] & 63 &  3 & 2646 &  6  & $8$  & 6 & 12 & SS  \\  \hline
ATD[68, 1] & 68 &  4 & 2312 &  8  & $8$  & 8 & 10 & SS  \\  \hline
ATD[72, 1] & 72 &  2 & 5184 &  4  & $8$  & 8 & 12 & SS  \\  \hline
ATD[78, 1] & 78 &  6 & 2028 &  12  & $8$  & 8 & 10 & SS  \\  \hline
ATD[78, 3] & 78 &  6 & 2028 &  12  & $8$  & 8 & 10 & SS  \\  \hline
ATD[80, 1] & 80 &  4 & 3200 &  8  & $8$  & 8 & 12 & SS  \\  \hline
ATD[80, 3] & 80 &  4 & 3200 &  8  & $8$  & 8 & 12 & SS  \\  \hline
ATD[81, 1] & 81 &  9 & 1458 &  18  & $8$  & 8 & 10 & SS  \\  \hline
ATD[81, 3] & 81 &  3 & 4374 &  6  & $8$  & 8 & 13 & SS  \\  \hline
ATD[84, 1] & 84 &  3 & 4704 &  6  & $8$  & 6 & 14 & SS  \\  \hline
ATD[84, 3] & 84 &  12 & 1176 &  24  & $8$  & 8 & 12 & SS  \\  \hline
ATD[84, 5] & 84 &  6 & 2352 &  12  & $8$  & 8 & 12 & SS  \\  \hline
ATD[84, 18] & 84 &  6 & 2352 &  12  & $8$  & 8 & 12 & SS  \\  \hline
ATD[93, 1] & 93 &  3 & 5766 &  6  & $8$  & 6 & 15 & SS  \\  \hline
ATD[100, 1] & 100 &  4 & 5000 &  8  & $8$  & 8 & 12 & SS  \\  \hline
\hline
\end{tabular}
\end{tiny}
\vspace{2mm}
\caption{Small examples of $\Delta\#\Delta^{-1}$ for $\Delta$ non-reversible.}
\label{T4}
\end{table}
\end{center}

\end{document}